\documentclass[12pt,amstex]{amsart}

\usepackage{mathptmx}
\usepackage{mathrsfs}

\usepackage{verbatim}
\usepackage{url}
\usepackage[all]{xy}
\usepackage{color}

\usepackage[colorlinks=true,citecolor=blue]{hyperref}

\usepackage{stmaryrd}
\usepackage{epsfig}
\usepackage{amsmath}
\usepackage{amssymb}
\usepackage{appendix}
\usepackage{amscd}
\usepackage{graphicx}
\usepackage{pstricks}

\usepackage[numbers,sort&compress]{natbib}
\theoremstyle{plain}
\newtheorem{thm}{Theorem}[section]
\newtheorem{lem}[thm]{Lemma}

\newtheorem{prop}[thm]{Proposition}

\newtheorem{cor}[thm]{Corollary}
\theoremstyle{definition}
\newtheorem{de}[thm]{Definition}

\theoremstyle{remark}
\newtheorem{rem}[thm]{Remark}

\def \N {\mathbb N}

\def \R {\mathbb R}

\def \U {\mathcal{U}}

\def \t {\mathcal{T}}

\def \U {\mathcal{U}}

\def \id {{\rm id}}

\def \a {\alpha }

\def \ep {\epsilon}

\begin{document}
	\title{Local weighted topological pressure}
	\author{Fangzhou Cai}
	\address[F. Cai]{School of Mathematics and Systems Science, Guangdong Polytechnic Normal University, Guangzhou, 510665, PR China}
	\email{cfz@mail.ustc.edu.cn}
		\maketitle
		\begin{abstract}
		In [D. Feng, W. Huang,  Variational principle for weighted topological pressure. J. Math. Pures Appl. (2016)], the authors studied    weighted topological pressure and established a	variational principle for it. In this paper, we  introduce the notion of local weighted topological pressure and generalize Feng and Huang's main results to localized version. 
		\end{abstract}
	\section{Introduction}	
	\subsection{Weighted topological entropy and pressure}We say that $(X,T)$ is a topological dynamical system (TDS for short) if $X$ is a compact
	metric space and $T$ is a continuous map from $X$ to $X$. We define $M(X)$ and $M(X,T)$ as the sets of Borel probability measures and  $T$-invariant Borel probability measures on $X$. We first briefly review the  classical theory of entropy and pressure in dynamical systems. One of the most basic theorems about  entropy is variational principle \cite{good69, din70, good71}:
	$$h_{top}(X, T) = \sup_{\mu\in M(X,T)}h_\mu(T),$$
	where $h_{top}(X, T)$ is the topological entropy of $(X,T)$ and $h_\mu(T)$ is the Kolomogorov–Sinai entropy.
	Motivated by statistical mechanics, Ruelle \cite{rue73}  and Walters \cite{wal75} introduced the topological pressure $P(T, f)$ for  a
	real-valued continuous function $f$ on $X$ and proved the variational principle:
$$P(T,f) = \sup_{\mu\in M(X,T)}\big(h_\mu(T)+\int_X fd\mu\big).$$	

Motivated by fractal geometry of self-affine carpets and sponges \cite{bed84, mc84, kp96a}, Feng and Huang \cite{fh} introduced weighted topological  pressure for factor maps
between dynamical systems, and established a variational principle for it. To be  precise, let us introduce
some notations first.
 Let $(X,T)$ and $(Y,S)$ be two TDSs. We say $(Y, S)$ is a factor of $(X, T)$ if there exists a continuous surjective
map $\pi:X\to Y$ such that $\pi\circ T=S\circ \pi$. The map $\pi$ is called a factor map from $X$ to $Y$. For $\mu\in M(X,T)$, we
denote by $\pi\mu\in M(Y,S)$ the push-forward of $\mu$ by $\pi.$ Denote the set of real-valued continuous functions  on $X$ by $C(X,\R)$.
Let $f\in C(X,\R)$ and $a_1>0,a_2\geq 0$. The main purpose of \cite{fh} is to consider 
the following question:

How can one define a meaningful term $P^{(a_1,a_2)}(T,f)$ such that the following variational
principle holds?
$$P^{(a_1,a_2)}(T, f) = \sup_{\mu\in M(X,T)}\Big(a_1h_\mu(T) + a_2h_{\pi\mu}(S)+\int_X fd\mu\Big).$$
Feng and Huang's definition for $P^{(a_1,a_2)}(T,f)$  was inspired from the dimension theory of affine
invariant subsets of tori, and from the “dimension” approaches of Bowen \cite{bowen} and Pesin–Pitskel \cite{pp} in
defining the topological entropy and topological pressure for arbitrary subsets.
We will give the detailed definition in the next section.

In recent years, lots of authors focus on  weighted topological and measure-theoretic
entropy and pressure. In \cite{wh}, Wang and Huang introduced various weighted topological
entropies from different points of view and studied their relationships.     The notion of measure-theoretic weighted entropy was also defined and studied in \cite{wh}. In \cite{zczy} the authors studied weighted topological entropy of the set of generic points. It was proved in \cite{zczy} that the weighted
topological entropy of generic points of the ergodic measure $\mu$ is equal to the weighted
measure entropy of $\mu$, which generalized the classical result of Bowen \cite{bowen}.
In \cite{sxz}, the authors studied weighted entropy of a flow on non-compact sets. In \cite{yclz}, the authors studied weighted topological entropy of random dynamical systems.
Recently, Tsukamoto \cite{tm} introduced a new approach to weighted topological  pressure and established a corresponding variational principle for it. The new approach is a modification of the familiar definition of topological entropy and pressure. It is very different from the original definitions in \cite{fh}. The equivalence
of the two definitions is highly nontrivial. In \cite{ycyy}, the authors generalized Tsukamoto's approach to amenable group action.
\subsection{Local entropy and pressure}
The local theory of entropy and pressure  is fundamental to many areas in dynamical systems. It has relations with entropy pairs, entropy sets, entropy points, and entropy structure, etc. (please see \cite{4,5,6,7,10,13,14,15,ro,hmry,hyz,hy}). In this subsection we briefly review the local theory of entropy and pressure. With the notion of entropy pairs \cite{5,7} in both topological and measure-theoretic situations, 
the study of local version of the variational principle of entropy has attracted a lot of attention. Blanchard, Glasner and Host \cite{6} showed the following
local variational principle: for an open cover $\U$ of $X$, there exists $\mu\in M(X,T)$ such that
$$\inf_{\a\succeq \U}h_{\mu}(T,\a)\geq h_{top}(T,\U),$$
where the infimum is taken over all  $\a\in \mathcal{P}_X,$ i.e., all finite Borel partitions of $X$, finer than $\U$. To make a general investigation on the converse of the inequality,
Romagnoli \cite{ro} introduced two types of measure-theoretic entropies related to $\U$: 
$$h_{\mu}(T,\U)=\lim_{n\to\infty}\frac{1}{n}\inf_{\a\succeq \U_0^{n-1}}H_{\mu}(\a),$$
$$h_{\mu}^+(T,\U)=\inf_{\a\succeq \U}h_{\mu}(T,\a).$$
 He showed the following local variational principle:
$$h_{top}(T,\U)=\max_{\mu\in M(X,T)}h_{\mu}(T,\U),$$
and the supremum can be attained by an ergodic measure.
Later, Glasner and Weiss \cite{13} proved that if $(X,T)$ is
invertible then the local variational principle also holds for 
$h_{\mu}^+(T,\U):$
$$h_{top}(T,\U)=\max_{\mu\in M(X,T)}h^+_{\mu}(T,\U).$$
Huang and Yi \cite{hy} generalized the above local variational
principles of entropy to the case of pressure: For any $f\in C(X,\R)$, the local pressure $P(T,f,\U)$  satisfies
$$P(T,f,\U)=\max_{\mu\in M(X,T)}\Big(h_\mu(T,\U)+\int_X fd\mu\Big),$$
and the supremum can be attained by an ergodic measure.
	The relative
	local variational principle for entropy was proved by Huang, Ye and Zhang in \cite{hyz}.
In \cite{wu},  Wu studied	various notions of local pressure of subsets and measures, which were defined by Carathéordory–Pesin construction.
	
	\medskip
	
In this paper,  we study the local weighted topological  pressure  with respect
to a fixed open cover. We generalize the main results in \cite{fh} to localized version. As a corallary, we get the variational principle for weighted topological  pressure, which was first proved in \cite{fh}.
\section{Local Weighted topological  pressure}
In this section, we introduce the notion of local weighted  topological  pressure and give some basic properties of it. First let us recall the definition of weighted  topological  pressure introduced in \cite{fh}.

	Let $k\geq 1,$ ${\bf a}=(a_1,\ldots,a_k)$ with $a_1>0$ and $a_i\geq 0,i=2,\ldots,k$. Let $(X_i,T_i), i=1,\ldots,k$ be TDSs with metric $d_i$ and $(X_{i+1},T_{i+1})$ be a factor  of $(X_i,T_i)$ for each $i$. Let $\tau_i:X_1\to X_{i+1}$  be the factor map and set $\tau_0=\id_{X_1}.$
	\begin{de}\cite[{\bf a}-{\it weighted Bowen ball}]{fh}
 For $x\in X_1,n\in\N$ and $\ep>0,$ denote
		$$B_n^{\bf a}(x,\ep):=\{y\in X_1: d_i(T_i^j\tau_{i-1}x,T_i^j\tau_{i-1}y)<\ep,0\leq j\leq \lceil(a_1+\ldots+a_i)n\rceil-1,1\leq i\leq k\},$$
		where $\lceil u\rceil$ denotes the least integer $\geq u.$
\end{de}
 For $f\in C(X_1,\R)$, denote $S_nf=\sum_{i=0}^{n-1}f\circ T_1^i.$
 Now we  give the definition
of {\bf a}-weighted topological pressure.
\begin{de}\cite[{\bf a}-{\it weighted topological pressure}]{fh}
Let  $Z\subset X_1, N\in\N, s\geq 0,$ $\ep>0$ and $f\in C(X_1,\R)$.
Define $$\Lambda^{{\bf a},s}_{N,\ep}(Z,f)=\inf\sum_{j}e^{-sn_j+\frac{1}{a_1}\sup_{x\in A_j}S_{\lceil a_1n_j\rceil}f(x)},$$
where the infimum is taken over all countable collections $\{(n_j,A_j)\}_j$ such that  $n_j\geq N, A_j$ be a Borel subset of $B_{n_j}^{\bf a}(x_j,\ep)$ for some $x_j\in X_1$ and $Z\subset \bigcup_jA_j.$
Define
$$\Lambda^{{\bf a},s}_{\ep}(Z,f)=\lim_{N\to \infty}\Lambda^{{\bf a},s}_{N,\ep}(Z,f)$$
and \begin{equation*}
	\begin{split}
	P^{{\bf a}}(T_1,Z,\ep,f)&=\inf\{s:\Lambda^{{\bf a},s}_{\ep}(Z,f)=0\}\\
	&=\sup\{s:\Lambda^{{\bf a},s}_{\ep}(Z,f)=+\infty\}.	
	\end{split}
\end{equation*} The {\it {\bf a}-weighted topological pressure for $Z$ with potential $f$} is defined as
$$P^{{\bf a}}(T_1,Z,f)=\lim_{\ep\to 0}P^{{\bf a}}(T_1,Z,\ep,f).$$

	For a function $g:X_1\to [0,+\infty)$, define
$$W^{{\bf a},s}_{N,\ep}(g)=\inf\sum_{j}c_je^{-sn_j+\frac{1}{a_1}\sup_{x\in A_j}S_{\lceil a_1n_j\rceil}f(x)},$$
where the infimum is taken over all countable collections $\{(n_j,A_j,c_j)\}_j$ such that  $n_j\geq N,$  $0<c_j<\infty$, $A_j$  be a Borel subset of $B_{n_j}^{\bf a}(x_j,\ep)$ for some $x_j\in X_1$ and $$\sum_{j}c_j\chi_{A_j}\geq g.$$
For $Z\subset X_1,$ we set $W^{{\bf a},s}_{N,\ep}(Z,f)=W^{{\bf a},s}_{N,\ep}(\chi_Z).$ Define
$$W^{{\bf a},s}_{\ep}(Z,f)=\lim_{N\to \infty}W^{{\bf a},s}_{N,\ep}(Z,f)$$
and \begin{equation*}
	\begin{split}
		P_{W}^{{\bf a}}(T_1,Z,\ep,f)&=\inf\{s:W^{{\bf a},s}_{\ep}(Z,f)=0\}\\
		&=\sup\{s:W^{{\bf a},s}_{\ep}(Z,f)=+\infty\}.
	\end{split}
\end{equation*}
The {\it average {\bf a}-weighted topological pressure  for  $Z$ with potential $f$} is defined as
 $$P_{W}^{{\bf a}}(T_1,Z,f)=\lim_{\ep\to 0}P_{W}^{{\bf a}}(T_1,Z,\ep,f).$$		
\end{de}
If there is no confusion, we omit $T_1,f$ and simply write $P^{{\bf a}}(Z),P_{W}^{{\bf a}}(Z)$ and $P^{{\bf a}}(Z,\ep),P_{W}^{{\bf a}}(Z,\ep)$ for short.

 It was proved in \cite{fh} that weighted topological pressure and  average weighted topological pressure are equal:
 \begin{prop}\cite[Proposition 3.5]{fh}\label{pro2.3}
 	For any $f\in C(X_1,\R)$ and $\ep>0$, we have$$P^{{\bf a}}(T_1,Z,6\ep,f)\leq P_{W}^{{\bf a}}(T_1,Z,\ep,f)\leq P^{{\bf a}}(T_1,Z,\ep,f),$$ hence $P^{{\bf a}}(Z,f)=P_{W}^{{\bf a}}(Z,f).$
 \end{prop}

\medskip

  We now define  local weighted topological pressure for a fixed open cover. The entropy case please also see \cite{zczy}.
\begin{de}
Let $Z\subset X_1, N\in\N, s\geq 0$, $f\in C(X_1,\R)$ and $\mathcal{U}_i$ be open covers of $X_i,i=1,\ldots,k$.
Define $$\Lambda^{{\bf a},s}_{N,\{\mathcal{U}_i\}_{i=1}^k}(Z,f)=\inf\sum_{j}e^{-sn_j+\frac{1}{a_1}\sup_{x\in A_j}S_{\lceil a_1n_j\rceil}f(x)},$$
where the infimum is taken over all countable collections $\{(n_j,A_j)\}_j$ such that  $n_j\geq N,$  $Z\subset \bigcup_jA_j$ and $A_j$ be a Borel subset of some element of $ \bigvee\limits_{i=1}^{k}(\tau_{i-1}^{-1}\mathcal{U}_i)_0^{\lceil(a_1+\ldots+a_i)n_j\rceil-1}.$ Since $\Lambda^{{\bf a},s}_{N,\{\mathcal{U}_i\}_{i=1}^k}(Z,f)$ does not decrease with $N$,
define
$$\Lambda^{{\bf a},s}_{\{\mathcal{U}_i\}_{i=1}^k}(Z,f)=\lim_{N\to \infty}\Lambda^{{\bf a},s}_{N,\{\mathcal{U}_i\}_{i=1}^k}(Z,f)$$
and
\begin{equation*}
	\begin{split}
	P^{{\bf a}}(T_1,Z,\{\mathcal{U}_i\}_{i=1}^k,f)&=\inf\{s:\Lambda^{{\bf a},s}_{\{\mathcal{U}_i\}_{i=1}^k}(Z,f)=0\}\\&=\sup\{s:\Lambda^{{\bf a},s}_{\{\mathcal{U}_i\}_{i=1}^k}(Z,f)=+\infty\}.
	\end{split}
\end{equation*}
 
For a function $g:X_1\to [0,+\infty)$, define $$W^{{\bf a},s}_{N,\{\mathcal{U}_i\}_{i=1}^k}(g)=\inf\sum_{j}c_je^{-sn_j+\frac{1}{a_1}\sup_{x\in A_j}S_{\lceil a_1n_j\rceil}f(x)},$$
	where the infimum is taken over all countable collections $\{(n_j,A_j,c_j)\}_j$ such that  $n_j\geq N,$ $0<c_j<\infty$, $A_j$  be a Borel subset of some element of $ \bigvee\limits_{i=1}^{k}(\tau_{i-1}^{-1}\mathcal{U}_i)_0^{\lceil(a_1+\ldots+a_i)n_j\rceil-1}$ and $\sum_{j}c_j\chi_{A_j}\geq g.$ 
	For $Z\subset X_1,$ we set $$W^{{\bf a},s}_{N,\{\mathcal{U}_i\}_{i=1}^k}(Z,f)=W^{{\bf a},s}_{N,\{\mathcal{U}_i\}_{i=1}^k}(\chi_Z).$$ Since $W^{{\bf a},s}_{N,\{\mathcal{U}_i\}_{i=1}^k}(Z,f)$ does not decrease with $N$, define
	$$W^{{\bf a},s}_{\{\mathcal{U}_i\}_{i=1}^k}(Z,f)=\lim_{N\to \infty}W^{{\bf a},s}_{N,\{\mathcal{U}_i\}_{i=1}^k}(Z,f)$$
and
\begin{equation*}
	\begin{split}
	P_{W}^{{\bf a}}(T_1,Z,\{\mathcal{U}_i\}_{i=1}^k,f)&=\inf\{s:W^{{\bf a},s}_{\{\mathcal{U}_i\}_{i=1}^k}(Z,f)=0\}\\&=\sup\{s:W^{{\bf a},s}_{\{\mathcal{U}_i\}_{i=1}^k}(Z,f)=+\infty\}. 
	\end{split}
\end{equation*}
 
\end{de}
If there is no confusion, we omit $T_1,f$ and simply write  $P^{{\bf a}}(Z,\{\mathcal{U}_i\}_{i=1}^k),P_{W}^{{\bf a}}(Z,\{\mathcal{U}_i\}_{i=1}^k)$ for short.
If $Z=X_1$, we also  write $P^{{\bf a}}(\{\mathcal{U}_i\}_{i=1}^k)$ and $P_{W}^{{\bf a}}(\{\mathcal{U}_i\}_{i=1}^k)$ instead of $P^{{\bf a}}(X_1,\{\mathcal{U}_i\}_{i=1}^k)$ and $P_{W}^{{\bf a}}(X_1,\{\mathcal{U}_i\}_{i=1}^k).$

\medskip

For an open cover $\U,$ denote $diam\U=\max_{U\in\U}diam U.$ The relation between weighted topological pressure and local weighted topological pressure is the following:
\begin{thm}\label{2.4} For subset $Z\subset X_1,$ we have
	\begin{equation*}
		\begin{split}
		P^{{\bf a}}(Z)	&=	\sup_{\{\mathcal{U}_i\}_{i=1}^k}P^{{\bf a}}(Z,\{\mathcal{U}_i\}_{i=1}^k)
			=\lim_{\max\limits_{1\leq i\leq k} diam\U_i\to 0}P^{{\bf a}}(Z,\{\mathcal{U}_i\}_{i=1}^k)\\
		&=P_W^{{\bf a}}(Z)=	\sup_{\{\mathcal{U}_i\}_{i=1}^k}P_W^{{\bf a}}(Z,\{\mathcal{U}_i\}_{i=1}^k)
			=\lim_{\max\limits_{1\leq i\leq k} diam\U_i\to 0}P_W^{{\bf a}}(Z,\{\mathcal{U}_i\}_{i=1}^k),
		\end{split}
	\end{equation*}	
where the supremum are taken over all open covers $\U_i$ of $X_i, i=1,\ldots,k.$
\end{thm} 
\begin{proof}
	By Proposition \ref{pro2.3} we have $	P^{{\bf a}}(Z)=P_W^{{\bf a}}(Z).$
We prove the  equality for $P^{{\bf a}}(Z)$, the one for $	P_W^{{\bf a}}(Z)$ is similar. 	Let $\U_i$ be open covers of $X_i$ and the Lebesgue number of $\U_i$ be $\ep_i, 1\leq i\leq k$. Let $\ep<\min_{i}\frac{\ep_i}{2}.$ Then 
	$$P^{{\bf a}}(Z,\{\mathcal{U}_i\}_{i=1}^k)\leq P^{{\bf a}}(Z,\ep).$$
	Hence $$	\sup_{\{\mathcal{U}_i\}_{i=1}^k}P^{{\bf a}}(Z,\{\mathcal{U}_i\}_{i=1}^k)\leq P^{{\bf a}}(Z).$$
	Let $\ep>0.$ If $\max\limits_{1\leq i\leq k} diam\U_i<\ep$, then $$P^{{\bf a}}(Z,\ep)\leq P^{{\bf a}}(Z,\{\mathcal{U}_i\}_{i=1}^k)\leq \sup_{\{\mathcal{U}_i\}_{i=1}^k}P^{{\bf a}}(Z,\{\mathcal{U}_i\}_{i=1}^k).$$
	Hence 
	$$P^{{\bf a}}(Z)\leq \lim_{\max\limits_{1\leq i\leq k} diam\U_i\to 0}P^{{\bf a}}(Z,\{\mathcal{U}_i\}_{i=1}^k)\leq \sup_{\{\mathcal{U}_i\}_{i=1}^k}P^{{\bf a}}(Z,\{\mathcal{U}_i\}_{i=1}^k).$$
\end{proof} 
For $M\in\N,$ $\U_i$ be open cover of $X_i$ and $f\in C(X_1,\R)$, recall $S_Mf=\sum_{i=0}^{M-1}f\circ T_1^i$ and $(\mathcal{U}_i)_0^{M-1}=\bigvee\limits_{l=0}^{M-1}T_i^{-l}\mathcal{U}_i.$
 We have the following proposition:
\begin{prop}\label{propp} Let $M\in \N, Z\subset X_1, f\in C(X_1,\R)$ and $\U_i$ be open covers of $X_i, i=1,\ldots,k$. We have
	$$P^{{\bf a}}(T_1,Z,\{\mathcal{U}_i\}_{i=1}^k,f)=\frac{1}{M}P^{{\bf a}}(T_1^M,Z,\{(\mathcal{U}_i)_0^{M-1}\}_{i=1}^k, S_Mf),$$
	$$P_{W}^{{\bf a}}(T_1,Z,\{\mathcal{U}_i\}_{i=1}^k,f)=\frac{1}{M}P_{W}^{{\bf a}}(T_1^M,Z,\{(\mathcal{U}_i)_0^{M-1}\}_{i=1}^k, S_Mf).$$

\end{prop}
\begin{proof}
	We prove the  equality for $P^{{\bf a}}$, the one for $	P_W^{{\bf a}}$ is similar.
Let $$h=P^{{\bf a}}(T_1,Z,\{\mathcal{U}_i\}_{i=1}^k,f).$$ For any $s<Mh,$ we have
$$\Lambda^{{\bf a},\frac{s}{M}}_{\{\mathcal{U}_i\}_{i=1}^k}(T_1,Z,f)=+\infty,$$ hence there exists  some $N$ such that $$\Lambda^{{\bf a},\frac{s}{M}}_{N,\{\mathcal{U}_i\}_{i=1}^k}(T_1,Z,f)\geq e^{\frac{M||f||}{a_1}}.$$  Now we prove $$\Lambda^{{\bf a},s}_{N,\{(\mathcal{U}_i)_0^{M-1}\}_{i=1}^k}(T_1^M,Z,S_Mf)\geq1.$$ Consider $$\sum_{j}e^{-sn_j+\frac{1}{a_1}\sup_{x\in A_j}\sum_{i=0}^{{\lceil a_1n_j\rceil}-1}(S_Mf)\circ T_1^{Mi}(x)},$$ where
  $n_j\geq N,$  $Z\subset \bigcup_jA_j$ and $A_j$  is a Borel subset of some element of $$ \bigvee_{i=1}^{k}\bigvee_{l=0}^{\lceil(a_1+\ldots+a_i)n_j\rceil-1}T_1^{-Ml}\tau_{i-1}^{-1}(\mathcal{U}_i)_0^{M-1}=\bigvee_{i=1}^{k}\tau_{i-1}^{-1}(\mathcal{U}_i)_0^{\lceil(a_1+\ldots+a_i)n_j\rceil M-1}.$$
  Since $$\bigvee_{i=1}^{k}\tau_{i-1}^{-1}(\mathcal{U}_i)_0^{\lceil(a_1+\ldots+a_i)n_j\rceil M-1}\succeq\bigvee_{i=1}^{k}\tau_{i-1}^{-1}(\mathcal{U}_i)_0^{\lceil(a_1+\ldots+a_i)Mn_j\rceil -1}$$
  and $Mn_j>N,$ by  definition  we have $$\sum_{j}e^{-\frac{s}{M}Mn_j+\frac{1}{a_1}\sup_{x\in A_j}S_{\lceil a_1Mn_j\rceil}f(x)}\geq e^{\frac{M||f||}{a_1}}.$$
  It is easy to check that $$\sum_{j}e^{-\frac{s}{M}Mn_j+\frac{1}{a_1}\sup_{x\in A_j}S_{\lceil a_1Mn_j\rceil}f(x)}\leq \sum_{j}e^{-\frac{s}{M}Mn_j+\frac{M||f||}{a_1}+\frac{1}{a_1}\sup_{x\in A_j}S_{M\lceil a_1n_j\rceil}f(x)}.$$ It follows that 
  $$\Lambda^{{\bf a},s}_{N,\{(\mathcal{U}_i)_0^{M-1}\}_{i=1}^k}(T_1^M,Z,S_Mf)\geq1.$$
 Let $s\to Mh$ we have $$Mh\leq P^{{\bf a}}(T_1^M,Z,\{(\mathcal{U}_i)_0^{M-1}\}_{i=1}^k,S_Mf).$$
  
  Now we prove the opposite direction. Let $$h=P^{{\bf a}}(T_1^M,Z,\{(\mathcal{U}_i)_0^{M-1}\}_{i=1}^k,S_Mf).$$ For any $s<\frac{h}{M},$ we have $$\Lambda^{{\bf a}, Ms}_{\{(\mathcal{U}_i)_0^{M-1}\}_{i=1}^k}(T_1^M,Z,S_Mf)=+\infty.$$ There exists $N_0$ such that $$\Lambda^{{\bf a}, Ms}_{N_0,\{(\mathcal{U}_i)_0^{M-1}\}_{i=1}^k}(T_1^M,Z,S_Mf)\geq e^{Ms(\frac{1}{a_1}+1)+\frac{\lceil a_1M+M\rceil||f||}{a_1}}.$$  Let $N>M(N_0+1+\frac{1}{a_1})$. Now we prove that $$\Lambda^{{\bf a}, s}_{N,\{\mathcal{U}_i\}_{i=1}^k}(T_1,Z,f)\geq1.$$
  Consider $$\sum_{j}e^{-sn_j+\frac{1}{a_1}\sup_{x\in A_j}S_{\lceil a_1n_j\rceil}f(x)},$$ where
  $n_j\geq N,$  $Z\subset \bigcup_jA_j$ and $A_j$  is a Borel subset of some element of $\bigvee\limits_{i=1}^{k}\tau_{i-1}^{-1}(\mathcal{U}_i)_0^{\lceil(a_1+\ldots+a_i)n_j\rceil -1}.$
Let $$\tilde{n}_j=\lceil\frac{n_j}{M}-\frac{1}{a_1}\rceil-1\geq N_0.$$  Note that \begin{equation*}
	\begin{split}
	\bigvee_{i=1}^{k}\tau_{i-1}^{-1}(\mathcal{U}_i)_0^{\lceil(a_1+\ldots+a_i)n_j\rceil -1}&\succeq \bigvee_{i=1}^{k}\tau_{i-1}^{-1}(\mathcal{U}_i)_0^{\lceil(a_1+\ldots+a_i)\tilde{n}_j\rceil M-1}\\
	&=\bigvee_{i=1}^{k}\bigvee_{l=0}^{\lceil(a_1+\ldots+a_i)\tilde{n}_j\rceil -1}T^{-Ml}_1\tau_{i-1}^{-1}(\mathcal{U}_i)_0^{M-1}.
	\end{split}
\end{equation*} 
By definition  we have  $$\sum_{j}e^{-Ms\tilde{n}_j+\frac{1}{a_1}\sup_{x\in A_j}\sum_{i=0}^{{\lceil a_1\tilde{n}_j\rceil}-1}(S_Mf)\circ T_1^{Mi}(x)}\geq e^{Ms(\frac{1}{a_1}+1)+\frac{\lceil a_1M+M\rceil||f||}{a_1}}.$$
It is easy to check that \begin{equation*}
	\begin{split}
	\sum_{j}e^{-Ms\tilde{n}_j+\frac{1}{a_1}\sup_{x\in A_j}S_{M\lceil a_1\tilde{n}_j\rceil}f(x)}&\leq
	 \sum_{j}e^{-Ms(\frac{n_j}{M}-\frac{1}{a_1}-1)+\frac{\lceil a_1M+M\rceil||f||}{a_1}+\frac{1}{a_1}\sup_{x\in A_j}S_{\lceil a_1n_j\rceil}f(x)}\\
	 &=\sum_{j}e^{-sn_j+\frac{1}{a_1}\sup_{x\in A_j}S_{\lceil a_1n_j\rceil}f(x)}e^{Ms(\frac{1}{a_1}+1)+\frac{\lceil a_1M+M\rceil||f||}{a_1}}.	
	\end{split}
\end{equation*}
Hence $$\sum_{j}e^{-sn_j+\frac{1}{a_1}\sup_{x\in A_j}S_{\lceil a_1n_j\rceil}f(x)}\geq 1.$$ It follows that $$\Lambda^{{\bf a}, s}_{N,\{\mathcal{U}_i\}_{i=1}^k}(T_1,Z,f)\geq1.$$
Let $s\to\frac{h}{M}$, we have $$\frac{h}{M}\leq P^{{\bf a}}(T_1,Z,\{\mathcal{U}_i\}_{i=1}^k,f).$$
\end{proof}
\section{lower bound}
In order to prove the lower bound of the variational principle for weighted topological pressure, in \cite{fh} the authors  established  weighted version of Shannon–McMillan–Breiman theorem (\cite[Proposition A.2]{fh}) and weighted version of
Brin–Katok theorem (\cite[Theorem 4.1]{fh}). It was shown in \cite[Proposition 4.2]{fh} that 
$$P^{{\bf a}}(T_1,f)\geq \sum_{i=1}^{k}a_ih_{\tau_{i-1}\mu}(T_i)+\int fd\mu$$ for any  $\mu\in M(X_1,T_1)$ and $f\in C(X_1,\R)$.
In this section, we generalize the above lower bound part of the variational principle, with the notion of local weighted topological pressure. 

We first need some lemmas:
\begin{prop}\cite[Proposition A.2]{fh}\label{p2}
	Let $(X,T)$ be a TDS, $\mathcal{B}$ be the collection of all Borel sets of $X$, $\mu\in M(X,T)$ and $k\geq 1.$	Let $\a_1,\ldots,\a_k\in \mathcal{P}_X$ with $H_{\mu}(\a_i)<\infty,1\leq i\leq k$, and ${\bf a}=(a_1,\ldots,a_k)\in \R^k$ with $a_1>0$ and $a_i\geq 0, i=2,\ldots,k.$ Then
	$$\lim_{N\to +\infty}\frac{1}{N}I_\mu\big(\bigvee_{i=1}^{k}(\a_i)_0^{\lceil(a_1+\ldots+a_i)N\rceil-1}\big)(x)=\sum_{i=1}^{k}a_i\mathbb{E}_\mu(F_i|\mathcal{I}_\mu)(x)$$
	almost everywhere, where
	$$F_i(x)=I_{\mu}\Big(\bigvee_{j=i}^{k}\a_j|\bigvee_{n=1}^\infty T^{-n}(\bigvee_{j=i}^k\a_j)\Big)(x), i=1,\ldots,k$$
	and $\mathcal{I}_\mu=\{B\in \mathcal{B}: \mu(B\Delta T^{-1}B)=0\}$. In particular, if $T$ is ergodic, we have $$\lim_{N\to +\infty}\frac{1}{N}I_\mu\big(\bigvee_{i=1}^{k}(\a_i)_0^{\lceil(a_1+\ldots+a_i)N\rceil-1}\big)(x)=\sum_{i=1}^{k}a_ih_\mu(T,\bigvee_{j=i}^{k}\tau_{j-1}^{-1}\a_j).$$
\end{prop}
\begin{lem}\label{pp}
	Let $(X,T)$ be a TDS, $\a=\{A_1,\ldots,A_m\}\in \mathcal{P}_X,$  $\mu\in M(X,T)$ and $\delta>0$. Then  there exists  an open cover $\U=\{U_0,U_1,\ldots,U_m\}$ such that
	\begin{enumerate}
		\item $\mu(U_0)<\delta.$
		\item $\mu(A_i\Delta U_i)<\delta,1\leq i\leq m.$
		\item $U_1,\ldots,U_m$ are pairwise disjoint.
	\end{enumerate} 
\end{lem}
\begin{proof}
	We can find compact sets $K_i\subset A_i$  with $\mu(A_i-K_i)< \frac{\delta}{m}, 1\leq i\leq m.$  Since $K_1,\ldots,K_m$ are pairwise disjoint, we can find open sets 
	$K_i\subset U_i$ such that $U_1,\ldots,U_m$ are pairwise disjoint. It is easy to see that $\mu(\bigcup_{i=1}^mU_i)>1-\delta.$ Hence we can find an open set $U_0$
	such that $(\bigcup_{i=1}^mU_i)^c\subset U_0$ and $\mu(U_0)<\delta.$ It is easy to see that $A_i\Delta U_i\subset\bigcup_{i=1}^m(A_i-K_i)$ and hence $\mu(A_i\Delta U_i)<\delta,1\leq i\leq m.$
\end{proof}

For a finite set $M$, denote $|M|$ the cardinality of $M.$ The following theorem is the main result of this section:
\begin{thm}\label{main}
	Let  $\a_i\in \mathcal{P}_{X_i},1\leq i\leq k$ and $\mu\in M(X_1,T_1)$. Then for any $\ep>0$, there exist open covers $\U_i$  of $X_i$ such that $|\U_i|=|\a_i|+1$, $\tau_{i-1}\mu(\U_i\Delta\a_i)<\ep,1\leq i\leq k$ and
	\begin{equation*}
		\begin{split}
			P^{{\bf a}}(supp\mu,\{\mathcal{U}_i\}_{i=1}^k,f)
			\geq\sum_{i=1}^ka_ih_\mu(T_1,\bigvee_{j=i}^{k}\tau_{j-1}^{-1}\a_j)+\int fd\mu-\ep.	
		\end{split}
	\end{equation*} 	
\end{thm}
\begin{proof}
	step 1:   reduce to an ergodic measure.
	
	Assume $\a_i=\{A_1^i,\ldots,A_{m_i}^i\},1\leq i\leq k$. Set $M=\max_{1\leq i\leq k}m_i.$ By   \cite[Lemma 4.15]{wal},  there exists $\delta_1$ such that if $\a,\beta\in \mathcal{P}_{X_1}, |\a|,|\beta|\leq (M+1)^k$ and $\mu(\a\Delta\beta)<\delta_1,$ then $$|h_\mu(T_1,\a)-h_\mu(T_1,\beta)|<\ep.$$ Choose $\delta>0$ such that $$\delta<\frac{1}{k}\min\{\frac{\ep}{k\log (M+1)\sum_{i=1}^ka_i+||f||+\ep}, \frac{\ep}{2\log(M+1)}, \frac{\delta_1}{M+1},(\frac{\ep}{M+1})^{\frac{1}{2}},1\}.$$
	By Lemma \ref{pp}, for each $1\leq i\leq k$ there exists  an open cover $\U_i=\{U_0^i,U_1^i,\ldots,U_{m_i}^i\}$ such that
	\begin{itemize}
		\item $\tau_{i-1}\mu(U_0^i)<\delta^2.$
		\item $\tau_{i-1}\mu(A_l^i\Delta U_l^i)<\delta^2,1\leq l\leq m_i.$
		\item $U_1^i,\ldots,U_{m_i}^i$ are pairwise disjoint.
	\end{itemize}
	It is easy to see that $\tau_{i-1}\mu(\U_i\Delta\a_i)<\ep.$ Let $$\beta_i=\{(\bigcup_{l=1}^{m_i}U_l^i)^c,U_1^i,\ldots,U_{m_i}^i\}\in \mathcal{P}_{X_i}.$$
	Then $$\tau_{i-1}\mu(\a_i\Delta\beta_i)<(M+1)\delta$$ and hence $$\mu\Big((\bigvee_{j=i}^{k}\tau_{j-1}^{-1}\a_j)\Delta(\bigvee_{j=i}^{k}\tau_{j-1}^{-1}\beta_j)\Big)<\delta_1.$$ It follows that 
	\begin{equation}\label{eq1}
		\sum_{i=1}^ka_ih_\mu(T_1,\bigvee_{j=i}^{k}\tau_{j-1}^{-1}\beta_j)>\sum_{i=1}^ka_ih_\mu(T_1,\bigvee_{j=i}^{k}\tau_{j-1}^{-1}\a_j)-\sum_{i=1}^ka_i\ep.
	\end{equation} Let $$\mu=\int_{M^e(X_1,T_1)}\theta dm(\theta)$$ be the ergodic decomposition of $\mu$.
	Denote $$I_\mu=\sum_{i=1}^ka_ih_\mu(T_1,\bigvee_{j=i}^{k}\tau_{j-1}^{-1}\beta_j)+\int fd\mu$$ and 
	$$I_\theta=\sum_{i=1}^ka_ih_\theta(T_1,\bigvee_{j=i}^{k}\tau_{j-1}^{-1}\beta_j)+\int fd\theta$$ for $\theta\in M^e(X_1,T_1)$. Then 
	$$I_\mu=\int_{M^e(X_1,T_1)}I_\theta dm(\theta)=\int_{\{\theta:I_\theta>I_\mu-\ep\}}I_\theta dm(\theta)+\int_{\{\theta:I_\theta\leq I_\mu-\ep\}}I_\theta dm(\theta).$$
	Since $$|\bigvee_{j=i}^{k}\tau_{j-1}^{-1}\beta_j|\leq (M+1)^k,$$ we have $$I_\theta\leq k\log (M+1)\sum_{i=1}^ka_i+||f||.$$
	Hence $$\int_{\{\theta:I_\theta>I_\mu-\ep\}}I_\theta dm(\theta)\leq \big(k\log (M+1)\sum_{i=1}^ka_i+||f||\big)m(\{\theta:I_\theta>I_\mu-\ep\})$$
	and
	$$\int_{\{\theta:I_\theta\leq I_\mu-\ep\}}I_\theta dm(\theta)\leq (I_\mu-\ep) m(\{\theta:I_\theta\leq I_\mu-\ep\}).$$
	We get $$m(\{\theta:I_\theta>I_\mu-\ep\})\geq \frac{\ep}{k\log (M+1)\sum_{i=1}^ka_i+||f||+\ep}> k\delta.$$
	Since for each $1\leq i\leq k,$ $\tau_{i-1}\mu(U_0^i)<\delta^2,$ it is easy to see that  $$m(\{\theta:\theta(\tau_{i-1}^{-1}U_0^i)<\delta\})>1-\delta.$$ Hence there exists $\theta\in M^e(X_1,T_1)$
	such that \begin{equation}\label{eq2}
		\theta(\tau_{i-1}^{-1}U_0^i)<\delta,\ \forall\ 1\leq i\leq k 
	\end{equation} and
	\begin{equation}\label{eq3}
		\sum_{i=1}^ka_ih_\theta(T_1,\bigvee_{j=i}^{k}\tau_{j-1}^{-1}\beta_j)+\int fd\theta>\sum_{i=1}^ka_ih_\mu(T_1,\bigvee_{j=i}^{k}\tau_{j-1}^{-1}\beta_j)+\int fd\mu-\ep.
	\end{equation}

	step 2: deal with the ergodic measure $\theta$ and finish the proof.
	
	Here we use the method in \cite{pp} (please also see \cite{zczy}).
	For $y\in supp\theta, 1\leq i\leq k$ and $n\in\N,$ write $t^i_n(y)$ for the number of integers $0\leq j\leq \lceil(a_1+\ldots+a_i)n\rceil-1$ such that $$T^j_i\tau_{i-1}y\in U_0^i.$$   By the
	Birkhoff Ergodic Theorem and (\ref{eq2}),  there exist $N_1\in\N$ and  $B_1\subset supp\theta$ such that $\theta(B_1)>\frac{5}{6}$ and for all $y\in B_1,n\geq N_1$ and $1\leq i\leq k,$
	$$t^i_n(y)< 2\delta\lceil(a_1+\ldots+a_i)n\rceil.$$
	By  Proposition \ref{p2}, there exist $N_2\in\N$ and $B_2\subset supp\theta$ such that $\theta(B_2)>\frac{5}{6}$ and for all $y\in B_2,n\geq N_2,$
	$$\theta\big(\bigvee_{i=1}^{k}(\tau_{i-1}^{-1}\beta_i)_0^{\lceil(a_1+\ldots+a_i)n\rceil-1}(y)\big)\leq e^{-n\sum_{i=1}^ka_ih_\theta(T_1,\bigvee_{j=i}^{k}\tau_{j-1}^{-1}\beta_j)+n\ep}.$$
	By the Birkhoff Ergodic Theorem there exist $N_3\in\N$ and  $B_3\subset supp\theta$ such that $\theta(B_3)>\frac{5}{6}$ and for all $y\in B_3,n\geq N_3,$
	$$|\frac{1}{a_1n}S_{\lceil a_1n\rceil}f(y)-\int fd\theta|<\ep.$$
	Let $N=\max\{N_1,N_2,N_3\}$ and $B=B_1\cap B_2\cap B_3.$ We have $\theta(B)>\frac{1}{2}.$
	Let $$\lambda<\sum_{i=1}^ka_ih_\theta(T_1,\bigvee_{j=i}^{k}\tau_{j-1}^{-1}\beta_j)-\sum_{i=1}^k(a_1+\ldots+a_i)\ep-2\ep+\int fd\theta.$$
	We will prove $\Lambda^{{\bf a},\lambda}_{N,\{\mathcal{U}_i\}_{i=1}^k}(supp\theta)\geq\frac{1}{2}e^{-k\ep}.$
	Consider$$\sum_{j}e^{-\lambda n_j+\frac{1}{a_1}\sup_{x\in A_j}S_{\lceil a_1n_j\rceil}f(x)},$$
	where  $n_j\geq N,$  $A_j$ is a Borel subset of some element of $ \bigvee\limits_{i=1}^{k}(\tau_{i-1}^{-1}\mathcal{U}_i)_0^{\lceil(a_1+\ldots+a_i)n_j\rceil-1}$ and $supp\theta\subset \bigcup_jA_j$.
	For $l\geq N$, let $$\Gamma_l=\{(n_j,A_j):n_j=l, A_j\cap B\neq\emptyset\}$$  and $$Y_l=\bigcup_{(n_j,A_j)\in\Gamma_l}A_j.$$
	Denote $$L=\{C\in \bigvee_{i=1}^{k}(\tau_{i-1}^{-1}\beta_i)_0^{\lceil(a_1+\ldots+a_i)l\rceil-1}:C\cap Y_l\cap B\neq\emptyset\}.$$
	Note that $$\sum_{C\in L}\theta(C)\geq \theta(Y_l\cap B)$$ and $C\cap B_2\neq\emptyset$ for $C\in L.$
	We have $$|L|\geq \theta(Y_l\cap B)e^{l\sum_{i=1}^ka_ih_\theta(T_1,\bigvee_{j=i}^{k}\tau_{j-1}^{-1}\beta_j)-l\ep}.$$
	For $(n_j,A_j)\in \Gamma_l,$ by the definition of $B_1$ we have
	\begin{equation*}
		\begin{split}
			&|\{C\in\bigvee_{i=1}^{k}(\tau_{i-1}^{-1}\beta_i)_0^{\lceil(a_1+\ldots+a_i)l\rceil-1}:C\cap A_j\cap B\neq\emptyset\}|\\
			&\leq \prod_{i=1}^k(M+1)^{2\delta\lceil(a_1+\ldots+a_i)l\rceil}
			\leq e^{l\ep\sum_{i=1}^k(a_1+\ldots+a_i)+k\ep},
		\end{split}
	\end{equation*}
	the second inequality holds because $2\delta\log(M+1)<\ep.$
	We have $$|\Gamma_l|\geq\theta(Y_l\cap B)e^{l\sum_{i=1}^ka_ih_\theta(T_1,\bigvee_{j=i}^{k}\tau_{j-1}^{-1}\beta_j)-l\ep-l\ep\sum_{i=1}^k(a_1+\ldots+a_i)-k\ep}.$$
	For $(n_j,A_j)\in \Gamma_l, A_j\cap B_3\neq\emptyset,$ we have
	\begin{equation*}
		\begin{split}
			&\sum_{j}e^{-\lambda n_j+\frac{1}{a_1}\sup_{x\in A_j}S_{\lceil a_1n_j\rceil}f(x)}
			\\&\geq\sum_{l=N}^\infty\sum_{(n_j,A_j)\in\Gamma_l}e^{-\lambda l+\frac{1}{a_1}\sup_{x\in A_j}S_{\lceil a_1l\rceil}f(x)}\\
			&\geq\sum_{l=N}^\infty\sum_{(n_j,A_j)\in\Gamma_l}e^{-\lambda l+l\int fd\theta-l\ep}	\\
			&\geq\sum_{l=N}^\infty \theta(Y_l\cap B)e^{-\lambda l+l\int fd\theta+l\sum_{i=1}^ka_ih_\theta(T_1,\bigvee_{j=i}^{k}\tau_{j-1}^{-1}\beta_j)-2l\ep-l\ep\sum_{i=1}^k(a_1+\ldots+a_i)-k\ep}\\
			&\geq \sum_{l=N}^\infty \theta(Y_l\cap B)e^{-k\ep}\geq\theta(B)e^{-k\ep}\geq \frac{1}{2}e^{-k\ep}.
		\end{split}
	\end{equation*}
	Hence $\Lambda^{{\bf a},\lambda}_{N,\{\mathcal{U}_i\}_{i=1}^k}(supp\theta)\geq\frac{1}{2}e^{-k\ep}.$ It follows that \begin{equation}\label{eq4}
		P^{{\bf a}}(supp\theta,\{\mathcal{U}_i\}_{i=1}^k,f)\geq\sum_{i=1}^ka_ih_\theta(T_1,\bigvee_{j=i}^{k}\tau_{j-1}^{-1}\beta_j)-\sum_{i=1}^k(a_1+\ldots+a_i)\ep-2\ep+\int fd\theta.	
	\end{equation}
	Since $supp\theta\subset supp\mu,$ combining (\ref{eq1}), (\ref{eq3}) and (\ref{eq4}) we finish the proof.
\end{proof}
\begin{rem}
	In \cite{hmry}, the authors proved the following lemma:
	\begin{lem}\cite[Lemma 9]{hmry}
		Let $(X,T)$ be a TDS. For every $M\in\N, \ep>0$, there exists $\delta>0$ such that for every $M$-set measurable covers  $\U=\{U_1,\ldots,U_M\}, \mathcal{V}=\{V_1,\ldots,V_M\}$ of
		$X$ with $\mu(\U\Delta\mathcal{V})<\delta,$ one has
		$$|h^+_\mu(T,\U)-h^+_\mu(T,\mathcal{V})|<\ep.$$
	\end{lem}
	Hence in Theorem \ref{main}, we can further require the open covers $\U_i$ satisfying $$|h^+_{\tau_{i-1}\mu}(T_i,\U_i)-h_{\tau_{i-1}\mu}(T_i,\a_i)|<\ep$$ and 
	$$|\sum_{i=1}^ka_ih^+_\mu(T_1,\bigvee_{j=i}^{k}\tau_{j-1}^{-1}\U_j)-\sum_{i=1}^ka_ih_\mu(T_1,\bigvee_{j=i}^{k}\tau_{j-1}^{-1}\a_j)|<\ep.$$
\end{rem}
From Theorem \ref{main} we can get the lower bound part of the variational principle for weighted topological pressure:
\begin{cor}Let $f\in C(X_1,\R)$. We have
	$$P^{{\bf a}}(T_1,f)\geq\sup_{\mu\in M(X_1,T_1)}\big(\sum_{i=1}^{k}a_ih_{\tau_{i-1}\mu}(T_i)+\int fd\mu\big).$$
\end{cor}
\begin{proof}
		Let  $\a_i\in \mathcal{P}_{X_i},1\leq i\leq k,$  $\mu\in M(X_1,T_1)$ and $\ep>0$. By Theorem \ref{main} there exist open covers $\U_i$  of $X_i$ such that 
	\begin{equation*}
			P^{{\bf a}}(supp\mu,\{\mathcal{U}_i\}_{i=1}^k,f)
			\geq\sum_{i=1}^ka_ih_\mu(T_1,\bigvee_{j=i}^{k}\tau_{j-1}^{-1}\a_j)+\int fd\mu-\ep.	
	\end{equation*} 
By Theorem \ref{2.4}, we have $$P^{{\bf a}}(T_1,f)\geq P^{{\bf a}}(T_1,\{\mathcal{U}_i\}_{i=1}^k,f)\geq	P^{{\bf a}}(supp\mu,\{\mathcal{U}_i\}_{i=1}^k,f).$$
Since $\a_i$ and $\ep$ are chosen arbitrary, we have 
$$P^{{\bf a}}(T_1,f)\geq\sum_{i=1}^ka_ih_{\tau_{i-1}\mu}(T_i)+\int fd\mu.$$
\end{proof}
\section{upper bound}
In order to prove the upper bound of the variational principle for weighted topological pressure, the authors in \cite{fh} applied the techniques
in geometric measure theory. They introduced the notion of average weighted topological pressure to prove the  dynamical Frostman lemma (\cite[Lemma 3.3]{fh}), which played a key role in the proof of the upper bound part of the variational principle. In this section, we generalize the result of the upper bound part in \cite{fh} to local case.

First we need some lemmas:
\begin{lem}\cite[Proposition 6]{ro}\label{l1}
	Let $\pi:(X,T)\to (Y,S)$ be a factor map between two TDSs and $\U$ be an open cover of $Y$. Then for any $\mu\in M(X,T)$, we have $$h_\mu(T,\pi^{-1}\U)=h_{\pi\mu}(S,\U).$$
\end{lem}
\begin{lem}\cite{hmry,hyz,hy}\label{l2}
Let $(X,T)$ be a TDS, $\mu\in M(X,T)$ and  $\U$ be an open cover of $X$.	If $(X,T)$ is invertible, then $$h_{\mu}^+(T,\U)=h_{\mu}(T,\U).$$
\end{lem}
\begin{lem}\cite[Proposition 5]{hmry}\label{l5}
	Let $(X,T)$ be a TDS, $\mu\in M(X,T)$ and $\U$  be an open cover of $X$. Let $$\mu=\int_{M^e(X,T)}\theta dm(\theta)$$ be the ergodic decomposition of $\mu$. Then
	$$h_\mu(T,\U)=\int_{M^e(X,T)}h_\theta(T,\U)dm(\theta).$$
\end{lem}

\begin{lem}\cite[Lemma 2.4]{cfh}\label{5.3}
Let $(X,T)$ be a TDS, $\nu\in M(X)$ and $M\in\N.$ Suppose $\a\in \mathcal{P}_X$ and $|\a|\leq M.$ Then for any $n,l\in\N$ satisfying $n\geq 2l,$ we have
	$$\frac{1}{n}H_\nu(\bigvee_{i=0}^{n-1}T^{-i}\a)\leq \frac{1}{l}H_{\nu_n}(\bigvee_{i=0}^{l-1}T^{-i}\a)+\frac{2l}{n}\log M,$$ 
	where $\nu_n=\frac{1}{n}\sum_{i=0}^{n-1}T^i\nu.$
\end{lem}
\begin{lem}\cite[Lemma 5.1]{fh}\label{2}
Let $(X,T)$ be a TDS and $\mu\in M(X,T)$. Suppose $\a\in \mathcal{P}_X$ and  $|\a|\leq M$. Write $$h(n)=H_{\frac{1}{n}\sum_{i=0}^{n-1}T^i\mu}(\a),$$ $$h(n,m)=H_{\frac{1}{m}\sum_{i=n}^{m+n-1}T^i\mu}(\a).$$
	Then
	$$|h(n+1)-h(n)|\leq\frac{1}{n+1}\log(3M^2(n+1)).$$
	$$|h(n+m)-\frac{n}{n+m}h(n)-\frac{m}{n+m}h(n,m)|\leq \log 2.$$
\end{lem}
\begin{lem}\cite[Lemma 5.4]{fh}\label{l8}
	Let $p\in\N$ and $u_j:\N\to\R,j=1,\ldots,p$ be bounded functions with $$\lim_{n\to\infty}|u_j(n+1)-u_j(n)|=0.$$
	Then for any positive numbers $c_1,\ldots,c_p$ and $r_1,\ldots,r_p$, we have
	$$\varlimsup_{n\to\infty}\sum_{j=1}^p\big(u_j(\lceil c_jn\rceil)-u_j(\lceil r_jn\rceil)\big)\geq 0.$$
\end{lem}
\begin{lem}\cite[Lemma 2]{hmry}\label{l9}
Let $(X,T)$ be a TDS, $\U=\{U_1,\ldots,U_m\}$ be an open cover of $X$ and $G:\mathcal{P}_X\to \R$ be monotone in the sense that $G(\a)\geq G(\beta)$ whenever $\a\succeq\beta.$ Then 
	$$\inf_{\a\in\mathcal{P}_X,\a\succeq \U}G(\a)=\inf_{\a:\a=\{A_1,\ldots,A_m\}, A_i\subset U_i, 1\leq i\leq m}G(\a).$$
\end{lem}

\medskip

Now we pay  attention back  to weighted topological pressure. We apply the techniques
in geometric measure theory and get the following useful lemma, which can be seen as the local version of the dynamical Frostman lemma  \cite[Lemma 3.3]{fh}.
\begin{lem}\label{m}
	Let $s\geq 0, N\in\N$ and $\mathcal{U}_i$ be open covers of $X_i,i=1,\ldots,k$. Suppose that $$c=W^{{\bf a},s}_{N,\{\mathcal{U}_i\}_{i=1}^k}(X_1)>0,$$ then there is a Borel probability measure $\mu$ on $X_1$ such that for any $n\geq N$ and  $$A\in \bigvee_{i=1}^{k}(\tau_{i-1}^{-1}\mathcal{U}_i)_0^{\lceil(a_1+\ldots+a_i)n\rceil-1},$$we have $$\mu(A)\leq \frac{1}{c}e^{-ns+\frac{1}{a_1}\sup_{x\in A}S_{\lceil a_1n\rceil}f(x)}.$$
\end{lem}
\begin{proof}
	Define a function $p$ on $C(X_1)$ by $$p(g)=\frac{1}{c}W^{{\bf a},s}_{N,\{\mathcal{U}_i\}_{i=1}^k}(g).$$
	It is easy to verify that
	\begin{enumerate}
		\item $p(g_1+g_2)\leq p(g_1)+p(g_2)$ for any $g_1,g_2\in C(X_1).$
		\item $p(tg)=tp(g)$ for any $t\geq 0, g\in C(X_1).$
		\item $p({\bf 1})=1, 0\leq p(g)\leq ||g||_\infty$ for any $g\in C(X_1),$ and $p(g)=0$ for $g\in C(X_1)$ with $g\leq 0.$
	\end{enumerate}
By the Hahn-Banach theorem, we can extend the linear functional $t\to tp({\bf 1}), t\in \R$, from the subspace of the constant functions to a linear functional $L:C(X_1)\to \R$ satisfying $L({\bf 1})=p({\bf 1})=1$ and $-p(-g)\leq L(g)\leq p(g)$ for any $g\in C(X_1).$
By the Riesz representation theorem we can  find a Borel probability measure $\mu$ on $X_1$ such that $L(g)=\int gdu$ for $g\in C(X_1).$ By standard discussion we can prove that
for any $n\geq N$ and any $A\in \bigvee_{i=1}^{k}(\tau_{i-1}^{-1}\mathcal{U}_i)_0^{\lceil(a_1+\ldots+a_i)n\rceil-1},$ we have $$\mu(A)\leq \frac{1}{c}e^{-ns+\frac{1}{a_1}\sup_{x\in A}S_{\lceil a_1n\rceil}f(x)}.$$
\end{proof}
For  open covers $\mathcal{U}_i$ of $X_i$ and $f\in C(X_1,\R)$, denote $$w_f(\bigvee_{i=1}^{k}\tau_{i-1}^{-1}\mathcal{U}_i)=\sup_{U\in\bigvee_{i=1}^{k}\tau_{i-1}^{-1}\mathcal{U}_i}\{|f(x)-f(y)|: x,y\in U\}.$$ The following theorem is the main result of this section:
\begin{thm}	\label{thm4.9}
	Let	$\mathcal{U}_i$ be open covers of $X_i,i=1,\ldots,k$ and $f\in C(X_1,\R)$.
	Then there exists an ergodic measure $\mu\in M(X_1,T_1)$ such that  $$ \sum_{i=1}^ka_ih_\mu(T_1,\bigvee_{j=i}^{k}\tau_{j-1}^{-1}\mathcal{U}_j)\geq P_{W}^{{\bf a}}(T_1,\{\mathcal{U}_i\}_{i=1}^k,f)-3w_f(\bigvee_{i=1}^{k}\tau_{i-1}^{-1}\mathcal{U}_i)-\int fd\mu.$$
\end{thm}
\begin{proof}
case 1: $(X_1,T_1)$ is invertible and zero-dimensional. We will prove there exists an ergodic measure $\mu\in M(X_1,T_1)$ such that 
$$\sum_{i=1}^ka_ih_\mu(T_1,\bigvee_{j=i}^{k}\tau_{j-1}^{-1}\mathcal{U}_j)\geq P_{W}^{{\bf a}}(T_1,\{\mathcal{U}_i\}_{i=1}^k,f)-2w_f(\bigvee_{i=1}^{k}\tau_{i-1}^{-1}\mathcal{U}_i)-\int fd\mu.$$
Write $$ t_0(n)=0,\  t_i(n)=\lceil(a_1+\ldots+a_i)n\rceil.$$
For an open cover $\U=\{U_1,\ldots,U_d\}$ of $X_1$, define $$\U^*=\{\a\in\mathcal{P}_{X_1}:\a=\{A_1,\ldots,A_d\}, A_m\subset U_m, A_m\text{\ are clopen \ sets,\ }m=1,\ldots,d\}.$$
Denote $h=P_{W}^{{\bf a}}(\{\mathcal{U}_i\}_{i=1}^k).$ Without loss of generality, we assume $h>0.$ Let  $$M=\max_{1\leq i\leq k} |\bigvee_{j=i}^{k}\tau_{j-1}^{-1}\mathcal{U}_j|.$$
{\bf Claim:} For every $\l\in\N,$ the set \begin{equation*}
	\begin{split}
		&M(l)=\{\mu\in M(X_1,T_1):	\forall\  \beta_i\in(\bigvee_{j=i}^{k}\tau_{j-1}^{-1}\mathcal{U}_j)^*,i=1,\ldots,k,
		\\&\sum_{i=1}^ka_iH_\mu(\bigvee_{h=0}^{l-1}T^{-h}\beta_i)\geq l\big(h-2w_f(\bigvee_{i=1}^{k}\tau_{i-1}^{-1}\mathcal{U}_i)-\int fd\mu\big)-1-k\sum_{i=1}^ka_i\log 2	
	\}
	\end{split}
\end{equation*} is not empty.

proof of claim: Fix $l\in\N$. Choose some $s\geq 0$ such that $h-\frac{1}{l}\leq s<h.$ By definition there exists $N\in\N$ such that $W^{{\bf a},s}_{N,\{\mathcal{U}_i\}_{i=1}^k}(X_1)\geq 1.$ By Lemma \ref{m}, there is a Borel probability measure $\nu$ on $X_1$ such that for any $n\geq N$ and  $A\in \bigvee\limits_{i=1}^{k}(\tau_{i-1}^{-1}\mathcal{U}_i)_0^{\lceil(a_1+\ldots+a_i)n\rceil-1},$we have $$\nu(A)\leq e^{-ns+\frac{1}{a_1}\sup_{x\in A}S_{\lceil a_1n\rceil}f(x)}.$$	Define $\nu_m=\frac{1}{m}\sum_{j=0}^{m-1}T^j\nu$ and $$  \omega_{i,n}=\frac{\sum_{j=t_{i-1}(n)}^{t_i(n)-1}T^j\nu}{t_i(n)-t_{i-1}(n)}$$ for $1\leq i\leq k$ with $ t_i(n)>t_{i-1}(n).$   
For a Borel probability measure $\theta$ on $X_1$ and $1\leq i\leq k$, define $$H_\theta(i)=\inf_{\beta\in(\bigvee_{j=i}^{k}\tau_{j-1}^{-1}\mathcal{U}_j)^*}H_{\theta}(\bigvee\limits_{j=0}^{l-1}T^{-j}\beta) .$$
For $$n> \max\{N,\{\frac{2l+1}{a_i}: 1\leq i\leq k, a_i\neq 0\} \}$$
and  $$\beta_i\in(\bigvee_{j=i}^{k}\tau_{j-1}^{-1}\mathcal{U}_j)^*,i=1,\ldots,k,$$ we have
\begin{equation*}
	\begin{split}
		\bigvee_{\substack{i:1\leq i\leq k,\\t_i(n)>t_{i-1}(n)}} (\beta_i)_{t_{i-1}(n)}^{t_i(n)-1}&\succeq \bigvee_{\substack{i:1\leq i\leq k,\\t_i(n)>t_{i-1}(n)}}(\bigvee_{j=i}^{k}\tau_{j-1}^{-1}\mathcal{U}_j)_{t_{i-1}(n)}^{t_i(n)-1}\\&=\bigvee_{i=1}^{k}(\tau_{i-1}^{-1}\mathcal{U}_i)_0^{\lceil(a_1+\ldots+a_i)n\rceil-1}.
	\end{split}
\end{equation*}
Hence for $A\in \bigvee\limits_{\substack{i:1\leq i\leq k,\\t_i(n)>t_{i-1}(n)}} (\beta_i)_{t_{i-1}(n)}^{t_i(n)-1}$, 
we have $$\nu(A)\leq e^{-ns+\frac{\lceil a_1n\rceil}{a_1}w_f(\bigvee_{i=1}^{k}\tau_{i-1}^{-1}\mathcal{U}_i)+\frac{1}{a_1}\sup_{x\in A}S_{\lceil a_1n\rceil}f(x)}.$$ It follows that $$-\nu(A)\log \nu(A)\geq \big(ns-\frac{\lceil a_1n\rceil}{a_1}w_f(\bigvee_{i=1}^{k}\tau_{i-1}^{-1}\mathcal{U}_i)-\frac{1}{a_1}\sup_{x\in A}S_{\lceil a_1n\rceil}f(x)\big)\nu(A)$$ and
 $$H_\nu\Big(\bigvee_{\substack{i:1\leq i\leq k,\\
 		t_i(n)>t_{i-1}(n)}}\bigvee_{j=t_{i-1}(n)}^{t_i(n)-1}T^{-j}\beta_i\Big)\geq ns-\frac{2\lceil a_1n\rceil}{a_1}w_f(\bigvee_{i=1}^{k}\tau_{i-1}^{-1}\mathcal{U}_i)-\frac{1}{a_1}\int S_{\lceil a_1n\rceil}fdv.$$	
For $i$ with  $1\leq i\leq k,a_i>0$, since $n\geq\frac{2l+1}{a_i}$, we have $t_i(n)-t_{i-1}(n)\geq 2l$, hence	by Lemma \ref{5.3}, we have
$$\frac{t_i(n)-t_{i-1}(n)}{l}H_{\omega_{i,n}}(\bigvee_{j=0}^{l-1}T^{-j}\beta_i)\geq H_{\nu}\Big(\bigvee_{j=t_{i-1}(n)}^{t_i(n)-1}T^{-j}\beta_i\Big)-2l\log M.$$
It follows that
\begin{equation*}
	\begin{split}
	&\sum_{i=1}^k\frac{t_i(n)-t_{i-1}(n)}{l}H_{\omega_{i,n}}(\bigvee_{j=0}^{l-1}T^{-j}\beta_i)\\
	&\geq ns-\frac{2\lceil a_1n\rceil}{a_1}w_f(\bigvee_{i=1}^{k}\tau_{i-1}^{-1}\mathcal{U}_i)-\frac{\lceil a_1n\rceil}{a_1}\int fd\nu_{t_1(n)}-2kl\log M.	
	\end{split}
\end{equation*}
Since $$\nu_{t_i(n)}=\frac{t_{i-1}(n)}{t_i(n)}\nu_{t_{i-1}(n)}+\frac{t_i(n)-t_{i-1}(n)}{t_i(n)}\omega_{i,n},$$	by Lemma \ref{2}, we have \begin{equation*}
		\begin{split}
	t_i(n)H_{\nu_{t_i(n)}}(\bigvee_{j=0}^{l-1}T^{-j}\beta_i)-t_{i-1}(n)H_{\nu_{t_{i-1}(n)}}(\bigvee_{j=0}^{l-1}T^{-j}\beta_i)\\
	\geq\big(t_i(n)-t_{i-1}(n)\big)H_{\omega_{i,n}}\big(\bigvee_{j=0}^{l-1}T^{-j}\beta_i\big)-t_i(n)\log 2.	
	\end{split}
	\end{equation*}
Hence 
\begin{equation*}
	\begin{split}
	&\sum_{i=1}^k	\Big(t_i(n)H_{\nu_{t_i(n)}}(\bigvee_{j=0}^{l-1}T^{-j}\beta_i)-t_{i-1}(n)H_{\nu_{t_{i-1}(n)}}(\bigvee_{j=0}^{l-1}T^{-j}\beta_i)\Big)\\
	&	\geq nls-\frac{2l\lceil a_1n\rceil}{a_1}w_f(\bigvee_{i=1}^{k}\tau_{i-1}^{-1}\mathcal{U}_i)-\frac{\lceil a_1n\rceil l}{a_1}\int fd\nu_{t_1(n)}-2kl^2\log M-kt_k(n)\log 2.	
	\end{split}
\end{equation*}
Since $\beta_i$ is chosen arbitrary, 
we have 
\begin{equation*}
	\begin{split}
\Theta_n:	&=	\sum_{i=1}^k	\Big(t_i(n)H_{\nu_{t_i(n)}}(i)-t_{i-1}(n)H_{\nu_{t_{i-1}(n)}}(i)\Big)\\
	&	\geq nls-\frac{2l\lceil a_1n\rceil}{a_1}w_f(\bigvee_{i=1}^{k}\tau_{i-1}^{-1}\mathcal{U}_i)-\frac{\lceil a_1n\rceil l}{a_1}\int fd\nu_{t_1(n)}-2kl^2\log M-kt_k(n)\log 2.	
	\end{split}
\end{equation*}
  Define  $$\gamma_n=\sum_{i=2}^kt_i(n)\big(H_{\nu_{t_i(n)}}(i)-H_{\nu_{t_1(n)}}(i)\big)-\sum_{i=2}^kt_{i-1}(n)\big(H_{\nu_{t_{i-1}(n)}}(i)-H_{\nu_{t_1(n)}}(i)\big).$$
 We have 
 $$\Theta_n=\gamma_n+\sum_{i=1}^k\big(t_i(n)-t_{i-1}(n)\big)H_{\nu_{t_1(n)}}(i).$$
 Define \begin{equation*}
 	\begin{split}
 	w(n)=&\sum_{i=2}^k(a_1+\ldots+a_{i-1})\big(H_{\nu_{t_{i-1}(n)}}(i)-H_{\nu_{t_1(n)}}(i)\big)\\
 	&-\sum_{i=2}^k(a_1+\ldots+a_i)\big(H_{\nu_{t_{i}(n)}}(i)-H_{\nu_{t_1(n)}}(i)\big).	
 	\end{split}
 \end{equation*}
In Lemma \ref{l8},  we take $p=2k-2,$
\begin{equation*}
	\begin{split}
	&u_j(n)=(a_1+\ldots+a_j)H_{\nu_{n}}(j+1), c_j=a_1+\ldots+a_j,  1\leq j\leq k-1;\\
	&u_j(n)=-(a_1+\ldots+a_{j-k+2})H_{\nu_{n}}(j-k+2), c_j=a_1+\ldots+a_{j-k+2},  k\leq j\leq 2k-2,	
	\end{split}
\end{equation*}

and take $r_j=a_1, 1\leq j\leq 2k-2.$ Since for any $1\leq i\leq k,$ $$|\bigvee\limits_{j=0}^{l-1}T^{-j}\beta|\leq M^l, \forall\beta\in(\bigvee\limits_{j=i}^{k}\tau_{j-1}^{-1}\mathcal{U}_j)^*,$$
by Lemma \ref{2} we have
\begin{equation*}
	|H_{\nu_n}(i)-H_{\nu_{n+1}}(i)|\leq\frac{1}{n+1}\log(3M^{2l}(n+1)).	
\end{equation*} 
Hence for any $1\leq j\leq 2k-2,$ $$\lim_{n\to\infty}|u_j(n+1)-u_j(n)|=0.$$ By Lemma \ref{l8} we have $\varlimsup\limits_{n\to\infty}w(n)\geq 0$. It follows that $$\varlimsup_{n\to\infty}-\frac{\gamma_n}{n}=\varlimsup_{n\to\infty}w(n)\geq 0.$$
Hence $$\varlimsup_{n\to\infty}\Big(\sum_{i=1}^ka_iH_{\nu_{t_1(n)}}(i)+l\int fd\nu_{t_1(n)}\Big)\geq l\big(s-2w_f(\bigvee_{i=1}^{k}\tau_{i-1}^{-1}\mathcal{U}_i)\big)-k\sum_{i=1}^ka_i\log 2.$$
 We can take a subsequence $(n_j)$ such that $$\lim_{j\to\infty}\Big(\sum_{i=1}^ka_iH_{\nu_{t_1(n_j)}}(i)+l\int fd\nu_{t_1(n_j)}\Big)=\varlimsup_{n\to\infty}\Big(\sum_{i=1}^ka_iH_{\nu_{t_1(n)}}(i)+l\int fd\nu_{t_1(n)}\Big)$$
 and $\nu_{t_1(n_j)}$ converges. Assume $\nu_{t_1(n_j)}\to\mu$. It is easy to see that $\mu$ is $T_1$-invariant. For any $\beta_i\in(\bigvee_{j=i}^{k}\tau_{j-1}^{-1}\mathcal{U}_j)^*,i=1,\ldots,k,$
  since $\beta_i$ are clopen sets, we have
  \begin{equation*}
  	\begin{split}
  	\sum_{i=1}^ka_iH_\mu(\bigvee_{h=0}^{l-1}T^{-h}\beta_i)&\geq l\big(s-2w_f(\bigvee_{i=1}^{k}\tau_{i-1}^{-1}\mathcal{U}_i\big)-\int fd\mu)-k\sum_{i=1}^ka_i\log 2\\
  	&\geq l\big(h-2w_f(\bigvee_{i=1}^{k}\tau_{i-1}^{-1}\mathcal{U}_i\big)-\int fd\mu)-1-k\sum_{i=1}^ka_i\log 2.	
  	\end{split}
  \end{equation*} 
  Hence $\mu\in M(l).$ The claim is proved.
  
  \medskip 

 It is easy to check that $M(l)$ is closed and $M(l_1l_2)\subset M(l_1)$ for $l_1,l_2\in\N$. Hence $\bigcap\limits_{l\in\N}M(l)\neq\emptyset.$ Let $\mu\in \bigcap\limits_{l\in\N}M(l).$
  We have $$\sum_{i=1}^ka_ih_\mu(T_1,\beta_i)\geq h-2w_f(\bigvee_{i=1}^{k}\tau_{i-1}^{-1}\mathcal{U}_i)-\int fd\mu, \ \forall  \beta_i\in(\bigvee_{j=i}^{k}\tau_{j-1}^{-1}\mathcal{U}_j)^*,i=1,\ldots,k.$$
  Since $X$ is zero-dimensional, there exists a fundamental base of the topology made of clopen sets. 
  Hence  by Lemma \ref{l9}, we have $$\sum_{i=1}^ka_ih^+_\mu(T_1,\bigvee_{j=i}^{k}\tau_{j-1}^{-1}\mathcal{U}_j)\geq h-2w_f(\bigvee_{i=1}^{k}\tau_{i-1}^{-1}\mathcal{U}_i)-\int fd\mu.$$
  Since $(X_1,T_1)$ is invertible, by Lemma \ref{l2} we have $h_\mu(T_1,\mathcal{U})=h^+_\mu(T_1,\mathcal{U}),$ hence (we also can prove  directly using Proposition \ref{propp},  see \cite[Proposition 4.3]{hy} ) $$\sum_{i=1}^ka_ih_\mu(T_1,\bigvee_{j=i}^{k}\tau_{j-1}^{-1}\mathcal{U}_j)\geq h-2w_f(\bigvee_{i=1}^{k}\tau_{i-1}^{-1}\mathcal{U}_i)-\int fd\mu.$$
Let $$\mu=\int_{M^e(X_1,T_1)}\theta dm(\theta)$$ be the ergodic decomposition of $\mu$. By Lemma \ref{l5}, there exists $\theta\in M^e(X_1,T_1)$ such that 
$$\sum_{i=1}^ka_ih_\theta(T_1,\bigvee_{j=i}^{k}\tau_{j-1}^{-1}\mathcal{U}_j)\geq P_{W}^{{\bf a}}(\{\mathcal{U}_i\}_{i=1}^k)-2w_f(\bigvee_{i=1}^{k}\tau_{i-1}^{-1}\mathcal{U}_i)-\int fd\theta.$$

\medskip

case 2: general case

	Let $(\tilde{X}_1,\sigma_{T_1})$ be the natural extension of $(X_1,T_1)$. That is:  $$\tilde{X}_1=\{(x_1,x_2,\ldots)\in X_1^\N: T_1(x_{i+1})=x_i,i\in\N\},$$
	 $\sigma_{T_1}:\tilde{X}_1\to\tilde{X}_1$ is defined as
	$$\sigma_{T_1}(x_1,x_2,\ldots)=(T_1x_1,x_1,x_2,\ldots).$$
	 Let $\pi:(\tilde{X}_1,\sigma_{T_1})\to (X_1,T_1)$ be the factor map which project each element of $\tilde{X}_1$ onto is first component.	
	Consider the following system $$\tilde{X}_1\to X_2\to\cdots\to X_k.$$
	By definition we have
	$$w_f(\bigvee_{i=1}^{k}\tau_{i-1}^{-1}\mathcal{U}_i)=w_{f\circ\pi}(\bigvee_{i=1}^{k}\pi^{-1}\tau_{i-1}^{-1}\mathcal{U}_i).$$
	It is easy to check that
	 $$P_{W}^{{\bf a}}(T_1,\{\mathcal{U}_i\}_{i=1}^k,f)-w_f(\bigvee_{i=1}^{k}\tau_{i-1}^{-1}\mathcal{U}_i)\leq P_{W}^{{\bf a}}(\sigma_{T_1},\{\pi^{-1}\mathcal{U}_1,\mathcal{U}_2,\ldots,\mathcal{U}_k\},f\circ\pi).$$ By case 1, there exists $\mu\in M^e(\tilde{X}_1,\sigma_{T_1})$ such that \begin{equation*}
	 	\begin{split}
	 	\sum_{i=1}^ka_ih_\mu(\sigma_{T_1},\bigvee_{j=i}^{k}\pi^{-1}\tau_{j-1}^{-1}\mathcal{U}_j)\geq &P_{W}^{{\bf a}}(\sigma_{T_1},\{\pi^{-1}\mathcal{U}_1,\mathcal{U}_2,\ldots,\mathcal{U}_k\},f\circ\pi)\\&-2w_{f\circ\pi}(\bigvee_{i=1}^{k}\pi^{-1}\tau_{i-1}^{-1}\mathcal{U}_i)-\int f\circ\pi d\mu\\
	 \geq 	&P_{W}^{{\bf a}}(T_1,\{\mathcal{U}_i\}_{i=1}^k,f)-3w_f(\bigvee_{i=1}^{k}\tau_{i-1}^{-1}\mathcal{U}_i)-\int f\circ\pi d\mu.	
	 	\end{split}
	 \end{equation*}
	 By Lemma \ref{l1}, we have $$\sum_{i=1}^ka_ih_{\pi\mu}(T_1,\bigvee_{j=i}^{k}\tau_{j-1}^{-1}\mathcal{U}_j)\geq P_{W}^{{\bf a}}(T_1,\{\mathcal{U}_i\}_{i=1}^k,f)-3w_f(\bigvee_{i=1}^{k}\tau_{i-1}^{-1}\mathcal{U}_i)-\int fd\pi\mu.$$
\end{proof}

By Theorem \ref{2.4} and Theorem \ref{thm4.9}   we know
$$\sup_{\{\mathcal{U}_i\}_{i=1}^k}\sup_{\mu\in M(X_1,T_1)}\Big(\sum_{i=1}^ka_ih_\mu(T_1,\bigvee_{j=i}^{k}\tau_{j-1}^{-1}\mathcal{U}_j)+\int fd\mu\Big)\geq P^{{\bf a}}(T_1,f).$$
We will show 	$$\sup\limits_{\{\mathcal{U}_i\}_{i=1}^k}\sum_{i=1}^{k}a_ih_\mu(T_1,\bigvee_{j=i}^{k}\tau_{j-1}^{-1}\U_j)=\sum_{i=1}^{k}a_ih_{\tau_{i-1}\mu}(T_i).$$
Hence the left side of the inequality above is just $$\sup_{\mu\in M(X_1,T_1)}\big(\sum_{i=1}^{k}a_ih_{\tau_{i-1}\mu}(T_i)+\int fd\mu\big).$$
We will give a more general result. 
We need the following lemma:
\begin{lem}\label{1}
	Let $\pi: (X,T)\to(Y,S)$ be a factor map between TDSs.  Let $\mu\in M(X,T)$, $\a=\{A_1,A_2,\ldots,A_k\}\in \mathcal{P}_Y$ and $\ep>0$. Then 	there exists an open cover $\mathcal{U}$ of $Y$ with $k$ elements such that for any $j\geq 0$ and any $\beta\in \mathcal{P}_X$ satisfying $T^{-j}\pi^{-1}\mathcal{U}\preceq \beta$, we have $H_{\mu}(T^{-j}\pi^{-1}\a|\beta)<\ep.$
\end{lem}
\begin{proof}
By \cite[Lemma 4.15]{wal} there exists $\delta>0$ such that whenever $\beta_1,\beta_2\in\mathcal{P}_X$ with $|\beta_1|=|\beta_2|=k$ and $\mu(\beta_1\Delta\beta_2)<\delta,$ then $H_\mu(\beta_1|\beta_2)<\ep$.
		Take closed subsets $B_i\subset A_i$ with $$\pi\mu(A_i-B_i)<\frac{\delta}{2k^2}, i=1,\ldots,k.$$ Let $B_0=(\bigcup_{i=1}^kB_i)^c$ and $U_i=B_0\cup B_i,i=1,\ldots,k.$ Then $\pi\mu(B_0)<\frac{\delta}{2k}$ and $\mathcal{U}=\{U_1,\ldots,U_k\}$ is an open cover of $Y$.
	For $j\geq 0$ and  $\beta\in\mathcal{P}_X$ satisfying $T^{-j}\pi^{-1}\mathcal{U}\preceq \beta,$ we can find  $\beta^\prime=\{C_1,\ldots,C_k\}\in\mathcal{P}_X$ satisfying $$C_i\subset T^{-j}\pi^{-1}U_i,i=1,\ldots,k$$ and $\beta\succeq\beta^\prime.$
	Since $$T^{-j}\pi^{-1}B_i\subset C_i\subset T^{-j}\pi^{-1}U_i,$$ we have
	$$\mu(C_i\Delta T^{-j}\pi^{-1}A_i)\leq \mu(T^{-j}\pi^{-1}A_i-T^{-j}\pi^{-1}B_i)+\mu(T^{-j}\pi^{-1}B_0)<\frac{\delta}{k}.$$
	Hence $$\sum_{i=1}^k\mu(C_i\Delta T^{-j}\pi^{-1}A_i)< \delta.$$ It follows that $H_{\mu}(T^{-j}\pi^{-1}\a|\beta^\prime)<\ep$ and hence $H_{\mu}(T^{-j}\pi^{-1}\a|\beta)<\ep.$
\end{proof}

We have  following theorem:
\begin{thm}
	For any $\mu\in M(X_1,T_1)$, we have
	 \begin{equation*}
		\begin{split}
		\sum_{i=1}^{k}a_i h_{\tau_{i-1}\mu}(T_i)&=\sup\limits_{\{\mathcal{U}_i\}_{i=1}^k}\varlimsup_{N\to +\infty}\frac{1}{N}	H_\mu(\bigvee_{i=1}^{k}(\tau_{i-1}^{-1}\U_i)_0^{\lceil(a_1+\ldots+a_i)N\rceil-1})\\
		&=\sup\limits_{\{\mathcal{U}_i\}_{i=1}^k}\varliminf_{N\to +\infty}\frac{1}{N}	H_\mu(\bigvee_{i=1}^{k}(\tau_{i-1}^{-1}\U_i)_0^{\lceil(a_1+\ldots+a_i)N\rceil-1}).		\end{split}
	\end{equation*}
where the supremum are taken over all open covers $\U_i$ of $X_i, i=1,\ldots,k.$
\end{thm}
\begin{proof}
	It is easy to see that
	\begin{equation*}\begin{split}
			\varlimsup_{N\to +\infty}\frac{1}{N}	H_\mu(\bigvee_{i=1}^{k}(\tau_{i-1}^{-1}\U_i)_0^{\lceil(a_1+\ldots+a_i)N\rceil-1})\leq \sum_{i=1}^{k}a_ih_\mu(T_1,\bigvee_{j=i}^{k}\tau_{j-1}^{-1}\U_j)\leq\sum_{i=1}^{k}a_i h_{\tau_{i-1}\mu}(T_i).
	\end{split}\end{equation*}	
Now we prove the opposite direction.	
	Let $\ep>0$ and $\a_i\in \mathcal{P}_{X_i},i=1,\ldots,k.$ By Lemma \ref{1}, we can find  corresponding  open covers $\mathcal{U}_i$ of $X_i$. For  $\beta\in \mathcal{P}_{X_1}$ with
	$$\beta\succeq\bigvee_{i=1}^{k}(\tau_{i-1}^{-1}\mathcal{U}_i)_0^{\lceil(a_1+\ldots+a_i)N\rceil-1},$$	
	we have $$\beta\succeq T_1^{-j}\tau_{i-1}^{-1}\mathcal{U}_i, 1\leq i\leq k, 0\leq j\leq \lceil(a_1+\ldots+a_i)N\rceil-1.$$
	By Lemma \ref{1}, we have $$H_\mu\big(T_1^{-j}\tau_{i-1}^{-1}\a_i|\beta\big)\leq\ep, 1\leq i\leq k, 0\leq j\leq \lceil(a_1+\ldots+a_i)N\rceil-1.$$
	It follows that
	\begin{equation*}
		\begin{split}
			H_\mu\big(\bigvee_{i=1}^{k}(\tau_{i-1}^{-1}\a_i)_0^{\lceil(a_1+\ldots+a_i)N\rceil-1}\big)&\leq H_\mu(\beta)+ H_\mu\big(\bigvee_{i=1}^{k}(\tau_{i-1}^{-1}\a_i)_0^{\lceil(a_1+\ldots+a_i)N\rceil-1}|\beta\big)\\
			&\leq 	H_\mu(\beta)+\sum_{i=1}^k\sum_{j=0}^{\lceil(a_1+\ldots+a_i)N\rceil-1}H_\mu\big(T_1^{-j}\tau_{i-1}^{-1}\a_i|\beta\big)\\
			&\leq 	H_\mu(\beta)+\sum_{i=1}^k\lceil(a_1+\ldots+a_i)N\rceil\ep.
		\end{split}
	\end{equation*}
	Since $\beta$ is arbitrary, we have \begin{equation*}
		\begin{split}
			&	H_\mu\big(\bigvee_{i=1}^{k}(\tau_{i-1}^{-1}\a_i)_0^{\lceil(a_1+\ldots+a_i)N\rceil-1}\big)
			\\
			&	\leq H_\mu\big(\bigvee_{i=1}^{k}(\tau_{i-1}^{-1}\mathcal{U}_i)_0^{\lceil(a_1+\ldots+a_i)N\rceil-1}\big)+\sum_{i=1}^k\lceil(a_1+\ldots+a_i)N\rceil\ep.
		\end{split}
	\end{equation*}
	Hence \begin{equation*}
		\begin{split}
		&\varliminf_{N\to +\infty}\frac{1}{N}	H_\mu\big(\bigvee_{i=1}^{k}(\tau_{i-1}^{-1}\a_i\big)_0^{\lceil(a_1+\ldots+a_i)N\rceil-1})\\&\leq \varliminf_{N\to +\infty}\frac{1}{N}	H_\mu\big(\bigvee_{i=1}^{k}(\tau_{i-1}^{-1}\U_i)_0^{\lceil(a_1+\ldots+a_i)N\rceil-1}\big)+\sum_{i=1}^k(a_1+\ldots+a_i)\ep.	
		\end{split}
	\end{equation*}
	Since $I_\mu\geq 0,$ by Fatou's Lemma  and Proposition \ref{p2}  we have
	\begin{equation*}
		\begin{split}
			\sum_{i=1}^{k}a_ih_\mu(T_1,\bigvee_{j=i}^{k}\tau_{j-1}^{-1}\a_j)&=\int \sum_{i=1}^{k}a_i\mathbb{E}_\mu(F_i|\mathcal{I}_\mu)(x)d\mu(x)\\ &=\int\varliminf_{N\to +\infty}\frac{1}{N}I_\mu\big(\bigvee_{i=1}^{k}(\tau_{i-1}^{-1}\a_i)_0^{\lceil(a_1+\ldots+a_i)N\rceil-1}\big)(x)d\mu(x) \\
			&\leq	\varliminf_{N\to +\infty}\int\frac{1}{N}I_\mu\big(\bigvee_{i=1}^{k}(\tau_{i-1}^{-1}\a_i)_0^{\lceil(a_1+\ldots+a_i)N\rceil-1}\big)(x)d\mu(x)\\
			&=
			\varliminf_{N\to +\infty}\frac{1}{N}	H_\mu\big(\bigvee_{i=1}^{k}(\tau_{i-1}^{-1}\a_i)_0^{\lceil(a_1+\ldots+a_i)N\rceil-1}\big)
		\end{split}
	\end{equation*}
	Hence $$ \sum_{i=1}^{k}a_ih_{\tau_{i-1}\mu}(T_i)\leq \sup\limits_{\{\mathcal{U}_i\}_{i=1}^k}\varliminf_{N\to +\infty}\frac{1}{N}	H_\mu\big(\bigvee_{i=1}^{k}(\tau_{i-1}^{-1}\U_i)_0^{\lceil(a_1+\ldots+a_i)N\rceil-1}\big).$$
\end{proof}
\begin{cor}For any $\mu\in M(X_1,T_1)$, we have
	$$\sup\limits_{\{\mathcal{U}_i\}_{i=1}^k}\sum_{i=1}^{k}a_ih_\mu(T_1,\bigvee_{j=i}^{k}\tau_{j-1}^{-1}\U_j)=\sum_{i=1}^{k}a_ih_{\tau_{i-1}\mu}(T_i).$$
\end{cor}
Combing Theorem \ref{main} and Theorem \ref{thm4.9}, we have
\begin{cor}\cite[Theorem 1.4]{fh}(variational principle for weighted topological pressure)
For $f\in C(X_1,\R),$ we have	$$P^{{\bf a}}(T_1,f)=\sup_{\mu\in M(X_1,T_1)}\Big(\sum_{i=1}^{k}a_ih_{\tau_{i-1}\mu}(T_i)+\int fd\mu\Big).$$
\end{cor}

\end{document}